\documentclass{tac}
\usepackage[all]{xy}
\usepackage{amsmath,amssymb}
\title{Entwining structures in monoidal categories}
\author{B. Mesablishvili}



\thanks{Supported by the research project "Algebraic and Topological Structures
in Homotopical and Categorical  Algebra, K-theory and Cyclic
Homology``, with financial support of the grant  GNSF/ST06/3-004.}
\eaddress{bachi@rmi.acnet.ge}

\newcommand{\A}{{\mathcal {A}}}
\newcommand{\B}{{\mathcal {B}}}
\newcommand{\T}{{\textbf{T}=(T, \eta, \mu)}}
\newcommand{\G}{{\textbf{G}=(G, \varepsilon, \delta)}}
\newcommand{\GG}{{\bar{{\textbf{G}}}=(\bar{G}, \bar{\varepsilon}, \bar{\delta})}}
\newcommand{\TT}{{\bar{\textbf{T}}=(\bar{T}, \bar{\eta}, \bar{\mu})}}
\newcommand{\V}{{\mathcal {V}}}
\newcommand{\la}{{\mathcal {(\mathbb{C},\mathbb{A},\lambda)}}}
\newcommand{\Al}{{ {\mathbb{A}=(A,e_A,m_A)}}}
\newcommand{\C}{{\mathcal {\mathbb{C}=(C,\varepsilon_C, \delta_C)}}}

\begin{document}
\maketitle
\begin{abstract}
Interpreting entwining structures as special instances of J. Beck's distributive law, the concept of entwining module can be generalized for the setting of arbitrary monoidal category. In this paper, we use the distributive law formalism to extend in this setting basic properties of entwining modules. 
\end{abstract}

\section{Introduction}

The important notion of entwining structures has been introduced by T. Brzezi\'nski and S. Majid in \cite{BM}. An entwining structure (over a commutative ring $K$) consists of a $K$-algebra $A$, a $K$-coalgebra $C$ and a certain $K$-homomorphism $\lambda: C \otimes_K A \to A \otimes_K C$ satisfying some axioms. Associated to $\lambda$ there is the category $\mathcal M _A ^C (\lambda)$ of entwining modules whose objects are at the same time  $A$-modules and $C$-comodules, with compatibility relation given by $\lambda$. 

The algebra $A$ can be identified with the monad $T= -\otimes_K A : \text{Mod}_K \to \text{Mod}_K$ whose Eilenberg-Moore category of algebras, $(\text{Mod}_K)^T$, is (isomorphic to) the category of right $A$-modules. Similarly, $C$ can be identified with the comonad $G= -\otimes_K C : \text{Mod}_K \to \text{Mod}_K$, and the corresponding Eilenberg-Moore category of coalgebras with the category of $C$-comodules. It turns out that to give an entwining structure $C \otimes_K A \to A \otimes_K C$ is to give a mixed distributive law $TG \to GT$ from the monad $T$  to the comonad G in the sense of J. Beck \cite{Bk}, which are in bijective corresondence with liftings (or extensions) $\overline{G}$ of the comonad $G$ to the category $(\text{Mod}_K)^T$; or, equivalently, liftings $\overline{T}$ of the  monad $T$ to the category $(\text{Mod}_K)_G$. Moreover, the categories $\mathcal M _A ^C (\lambda)$ , $((\text{Mod}_K)^T)_{\overline{}G}$ and $((\text{Mod}_K)_G)^T$ are isomorphic. Thus, the (mixed) distributive law formalism can be  used to study entwining structures and the corresponding category of modules. In this article -based on this formalism- we extend in the context of monoidal categories some of basic results on entwining structures that appear in the literature (see, for example, \cite{BW1}, \cite{C}, \cite{W1}).

The paper is organized as follows. After recalling the notion of Beck's mixed distributive law and the basic facts about it, we define in Section 3 an entwining structure in any monoidal category. In Section 4, we prove some categorical results that are needed in the next section, but may also be of independent interest. Finally, in the last section we present our main results.

We refer to M. Barr and C. Wells \cite{BW}, S. MacLane \cite{M} and F. Borceux \cite{B} for terminology and general results on (co)monads, and to T. Brzezi\'nski and R. Wisbauer \cite{BW1} for coring and comodule
theory.

\section{Mixed distributive laws}

Let $\T$ be a monad and $\G$ a comonad on a category $\A$. A
\emph{mixed distributive law} from $\T$ to $\G$ is a natural
transformation $$\lambda : \textbf{T} \textbf{G} \to
\textbf{G}\textbf{T}$$ for which the diagrams

$$
\xymatrix{ & G \ar[dl]_{\eta G} \ar[rd]^{G \eta}& & & T G
\ar[dl]_{\lambda} \ar[rd]^{T \varepsilon}&\\
T G \ar[rr]_{\lambda}&& G T \,,& G T \ar[rr]_{\varepsilon T} &&
T,}
$$

$$
\xymatrix{ T G \ar[d]_{\lambda} \ar[r]^{T\delta} & T G^2
\ar[r]^{\lambda G} & G T G \ar[d]^{G \lambda} &\text{and}& T^2 G
\ar[d]_{\mu G} \ar[r]^{T \lambda} & T G T \ar[r]^{\lambda T} & G T
T \ar[d]^{G
\mu}\\
G T \ar[rr]_{\delta T}&& G G T && TG \ar[rr]_{\lambda} && G T}
$$ commute.

Given a monad $\T$ on $\A$, write $\A ^{\textbf{T}}$ for the
Eilenberg-Moore category of $\textbf{T}$-algebras, and write
$F^{\textbf{T}} \dashv U^{\textbf{T}}: \A^{\textbf{T}} \to \A$ for
the corresponding forgetful-free adjunction. Dually, if $\G$ is
comonad on $\A$, then write $\A_{\textbf{G}}$ for the category of
$\textbf{G}$-coalgebras, and write $F_{\textbf{G}} \dashv
U_{\textbf{G}}: \A_{\textbf{G}} \to \A$ for the corresponding
forgetful-cofree adjunction.

\begin{theorem} \emph{(} see \cite{W} \emph{)} Let $\T$ be a monad and
$\G$ a comonad on a
category $\A$. Then the following structures are in bijective
correspondences:
\begin{itemize}
\item mixed distributive laws $\lambda : \textbf{T} \textbf{G} \to
\textbf{G}\textbf{T}$;

\item comonads $\GG$ on $\A^{T}$ that extend $\textbf{G}$ in the
sense that $U ^T \bar{G}= G U^T$, $U^T
\bar{\varepsilon}=\varepsilon U^T$ and $U^T \bar{\delta}=\delta
U^T$;

\item monads $\TT$ on $\A_{G}$ that extend $\textbf{T}$ in the
sense that $U_G \bar{T}=T U_G$, $U_G \bar{\eta}=\eta U_G$ and $U_G
\bar{\mu}=\mu U_G$.
\end{itemize} These correspondences are constructed as follows:
\begin{itemize}
\item Given a mixed distributive law $$\lambda : \textbf{T}
\textbf{G} \to \textbf{G}\textbf{T},$$ then $\bar{G}(a,
\xi_a)=(G(a), G(\xi_a) \cdot \lambda_a)$, $\bar{\varepsilon}_{(a,
\xi_a)}=\varepsilon_a$, $\bar{\delta}_{(a, \xi_a)}=\delta_a$, for
any $(a, \xi_a) \in \A ^\textbf{T}$; and $\bar{T}(a, \nu_a)=(T(a),
\lambda_a \cdot T(\nu_a))$, $\bar{\eta}_{(a, \nu_a)}=\eta_a$,
$\bar{\mu}_{(a, \nu_a)}=\mu_a $ for any $(a, \nu_a) \in
\A_{\textbf{G}}$.

\item If $\GG$ is a comonad on $\A ^{\textbf{T}}$ extending the
comonad $\G$, then the corresponding distributive law $$\lambda :
\textbf{T} \textbf{G} \to \textbf{G}\textbf{T}$$ is given by
$$\xymatrix{TG \ar[r]^-{TG \eta}& TGT=U^{T} F^{T}
G U^{T} F^{T}=U^T F^T U^T \bar{G}F^T \ar[rr]^-{U^T \varepsilon ^T
\bar{G}F^T} && U^T \bar{G} F^T=GU^TF^T=GT,}$$ where
$\varepsilon^{\textbf{T}}: F^T U^T \to 1$ is the counit of the
adjunction $F^{T} \dashv U^{T}$.

\item If $\TT$ is a monad on $\A_{\textbf{G}}$ extending $\T$,
then the corresponding mixed distributive law is given by
$$\xymatrix{TG =TU_GF_G =U_G \bar{T} F_G \ar[rr]^-{U_G \eta_G
\bar{T} F_G} &&U_GF_GU_G \bar{T} F_G=U_G F_G T U_G F_G =GTG
\ar[r]^-{GT \varepsilon} & GT},$$ where $\eta_G : 1 \to F_GU_G$ is
the unit of the adjunction $U_G \dashv F_G.$

\end{itemize}

\end{theorem}

It follows from this theorem that if $$\lambda : \textbf{T}
\textbf{G} \to \textbf{G}\textbf{T}$$ is a mixed distributive law,
then $(\A_{\textbf{G}})^{\bar{\textbf{T}}}=(\A
^{\textbf{T}})_{\bar{\textbf{G}}}$. We write $(\A ^{\textbf{T}}
_{\textbf{G}})(\lambda)$ for this category. An object of this
category is a three-tuple $(a, \xi_a, \nu_a)$, where $(a, \xi_a)
\in \A ^\textbf{T}$, $(a, \nu_a) \in \A_{\textbf{G}}$, for which
$G(\xi_a) \cdot \lambda_a \cdot T(\nu_a)=\nu_a \cdot \xi_a$. A
morphism $f: (a, \xi_a, \nu_a) \to (a', \xi'_a, \nu'_a)$ in $(\A
^{\textbf{T}} _{\textbf{G}})(\lambda)$ is a morphism $f: a \to a'$
in $\A$ such that $\xi'_a \cdot T(f)=f \cdot \xi_a $ and $\nu'_a
\cdot f=G(f) \cdot \nu_a$.

\section{Entwining structures in monoidal categories}

Let $\V=(V, \otimes, I)$ be a monoidal category with coequalizers
such that the tensor product preserves the coequalizer in both
variables. Then for all algebras $\Al$ and
$\mathbb{B}=(B, e_B, m_B)$ and all $M \in \V_A$, $N \in {_A \V_B}$
and $P \in {_B\V}$, the tensor product $M \otimes_A N$ exists and
the canonical morphism $(M \otimes_A N) \otimes _B P \to M
\otimes_A (M \otimes_B P)$ is an isomorphism. Using MacLane's
coherence theorem (see, \cite{M}, XI.5), we may assume without
loss of generality that $\V$ is strict.

It is well known that every algebra $\mathbb{A}=(A,e_A,m_A)$ in
$\V$ defines a monad $\textbf{T}_{\mathbb{A}}$ on $\V$ by

\begin{itemize}
\item $T_{\mathbb{A}}(X)=X\otimes A$, \item
$(\eta_{T_{\mathbb{A}}})_X =X \otimes e_A : X \to X \otimes A$,
\item $(\mu_{T_{\mathbb{A}}})_X=X \otimes m_A : X\otimes A \otimes
A \to X \otimes A $,
\end{itemize} and that $\V ^{\textbf{T}_{\mathbb{A}}}$ is
(isomorphic to) the category $\V _{\mathbb{A}}$ of right
$A$-modules.

\bigskip
Dually, if $\mathbb{C}=(C,\varepsilon_C, \delta_C,)$ is a
coalgebra (=comonoid) in $\V$, then one defines a comonad
$\textbf{G}_{\mathbb{C}}$ on $\V$ by
\begin{itemize}
\item $G _{\mathbb{C}}(X)=X \otimes C$, \item
$(\varepsilon_{G_{\mathbb{C}}})_X =X \otimes \varepsilon_C : X
\otimes C \to  X,$ \item $(\delta_{G_{\mathbb{C}}})_X =X \otimes
\delta_C : X \otimes C \to X \otimes C \otimes C$,

\end{itemize} and
$\V_{\textbf{G}_{\mathbb{C}}}$ is (isomorphic to) the category
$\V^{\mathbb{C}}$ of right $C$-comodules.

\bigskip

Quite obviously, if $\lambda$ is a mixed distributive law from
$\textbf{T}_{\mathbb{A}}$ to $\textbf{G}_{\mathbb{C}}$, then the
morphism
$$ \lambda'=\lambda_I : C \otimes A  \to  A \otimes C$$
makes the following diagrams commutative:

$$
\xymatrix{  C  \ar[d]_{C \otimes e_A} \ar[dr]^{e_A \otimes C} & &
& C \otimes A
\ar[d]_{\lambda'} \ar[rd]^-{\varepsilon_{\mathbb{C}} \otimes A}&\\
C \otimes A \ar[r]_{\lambda '}& A \otimes C \,,& & A \otimes C
\ar[r]_{A \otimes \varepsilon_{\mathbb{C}}} & A ,}
$$
$$
\xymatrix{ C \otimes A \ar[d]_-{\lambda'} \ar[r]^-{\delta_C
\otimes A} & C\otimes C \otimes A \ar[r]^{C \otimes \lambda' } &
C\otimes A \otimes C \ar[d]^{ \lambda' \otimes C} & C\otimes A
\otimes A \ar[d]_{C \otimes m_A} \ar[r]^{\lambda' \otimes A} & A
\otimes C \otimes A \ar[r]^{A \otimes \lambda'} & A \otimes A
\otimes C \ar[d]^{m_A \otimes C}\\
A \otimes C \ar[rr]_{A \otimes \delta_C}&& A \otimes C \otimes C,
& C \otimes A \ar[rr]_{\lambda'} && A \otimes C \,\,.}
$$ Conversely, if $\lambda' : C \otimes A \to A \otimes C$ is a
morphism for which the above diagrams commute, then the natural
transformation $$- \otimes \lambda' :
T_{\mathbb{A}}G_{\mathbb{C}}(-)=-\otimes C \otimes A \to -\otimes
A \otimes C=G_{\mathbb{C}}T_{\mathbb{A}}(-)$$ is a mixed
distributive law from the monad $\textbf{T}_{\mathbb{A}}$ to the
comonad $\textbf{G}_{\mathbb{C}}$. It is easy to see that
$\lambda' =(- \otimes \lambda')_I$. When $I$ is a regular
generator in $\V$ and the tensor product preserves all colimits in
both variables, it is not hard to show that $\lambda \simeq -
\otimes \lambda_I$. When this is the case, then the
correspondences $\lambda \to \lambda_I$ and $\lambda' \to -\otimes
\lambda'$ are inverses of each other.

\begin{definition} An entwining structure $(\mathbb{C},\mathbb{A},\lambda)$
consists of an algebra $\mathbb{A}=(A,e_A,m_A)$   and a coalgebra
$\mathbb{C}=(C,\varepsilon_C, \delta_C)$  in $\V$ and a morphism
$\lambda : C \otimes A \to A \otimes C$ such that the natural
transformation $$-\otimes \lambda
:T_{\mathbb{A}}G_{\mathbb{C}}(-)=-\otimes C \otimes A \to -\otimes
A \otimes C=G_{\mathbb{C}}T_{\mathbb{A}}(-)$$ is a mixed
distributive law from the monad $\textbf{T}_{\mathbb{A}}$ to the
comonad $\textbf{G}_{\mathbb{C}}$.
\end{definition}

Let be $\la$ be an entwining structure and let $\GG$ be the
comonad on $\V _{\mathbb{A}}$ that extends
$\textbf{G}=\textbf{G}_{\mathbb{C}}$. Then we know that, for any
$(V,\xi_V) \in \V _{\mathbb{A}}$,

$$
\xymatrix{ \bar{G}(V,\xi_V)=(V \otimes C, V \otimes C \otimes A
\ar[r]^-{V \otimes \lambda} & V \otimes A \otimes C \ar[r]^-{\xi_V
\otimes C} & V \otimes C) .}$$ In particular, since $(A, m_A) \in
\V _{\mathbb{A}}$, $A \otimes C$ is a right $A$-module with right
action
$$
\xymatrix{ \xi_{A \otimes C} : A \otimes C \otimes A \ar[r]^-{A
\otimes \lambda} & A \otimes A \otimes C \ar[r]^-{m_a \otimes C} &
A \otimes C.}$$

\begin{lemma} View $A \otimes C$ as a left $A$-module through
$\bar{\xi}_{A \otimes C}=m_A \otimes C$. Then $(A \otimes C,
\bar{\xi}_{A \otimes C}, \xi_{A \otimes C})$ is an
$A$-$A$-bimodule.
\end{lemma}
\begin{proof} Clearly $(A \otimes C, \bar{\xi}_{A \otimes C})
\in {_{\mathbb{A}}{\V}}.$ Moreover, since $(A \otimes
\lambda)\cdot (m_A \otimes C \otimes A)= (m_A \otimes A \otimes
C)\cdot (A \otimes A \otimes \lambda)$, it follows from the
associativity of $m_A$ that the diagram
$$
\xymatrix{A \otimes A \otimes C \otimes A \ar[r]^-{A \otimes A
\otimes \lambda} \ar[d]_{m_A \otimes  C \otimes A} & A \otimes A
\otimes A \otimes C \ar[d]^{A \otimes m_A \otimes C }\\
A \otimes C \otimes A \ar[d]_{A \otimes \lambda} & A \otimes A
\otimes C \ar[d]^{m_A \otimes C }\\
A \otimes A \otimes C  \ar[r]_{m_A \otimes C }& A \otimes C }$$ is
commutative, which just means that $(A \otimes C, \bar{\xi}_{A
\otimes C}, \xi_{A \otimes C})$ is an $A$-$A$-bimodule.
\end{proof}

\bigskip

Since $\bar{\varepsilon}_{(A,m_A)} : \bar{G}(A,m_A) \to (A,m_A)$
and $\bar{\delta}_{(A,m_A)}: \bar{G}(A,m_A) \to \bar{G}^2 (A,m_A)$
are morphisms of right $A$-modules, and since
$U_{\mathbb{A}}({\bar{\varepsilon}}_{(A,m_A)})=\varepsilon_A =$ \(
( A \otimes C \stackrel{{A \otimes
\varepsilon_{\mathbb{C}}}}{\longrightarrow}  A) \) and $U_A
(\bar{\delta}_{(A,m_A)})=\delta_{\mathbb{C}}=$\( (A \otimes C
\stackrel{A \otimes\delta_{\mathbb{C}}}{\longrightarrow} A \otimes
C \otimes C) \), it follows that \( A \otimes C \stackrel{{A
\otimes \varepsilon_{\mathbb{C}}}}{\longrightarrow} A \) and \( A
\otimes C \stackrel{A \otimes\delta_{\mathbb{C}}}{\longrightarrow}
A \otimes C \otimes C \) are both morphisms of right $A$-modules.
Clearly they are also morphisms of left $A$-modules with the
obvious left $A$-module structures arising from the multiplication
$m_A : A\otimes A \to A$, and hence morphisms of
$A$-$A$-bimodules. Since $ \C$ is a coalgebra in $\V$, it follows
that the triple $(\underline{A \otimes C})_{\lambda}=(A \otimes C
, \varepsilon_{(\underline{A \otimes C})_{\lambda}},
\delta_{(\underline{A \otimes C})_{\lambda}})$, where
$\varepsilon_{(\underline{A \otimes C})_{\lambda}}=$ \( A \otimes
C \stackrel{{A \otimes \varepsilon_{\mathbb{C}}}}{\longrightarrow}
A \) and $\delta_{(\underline{A \otimes C})_{\lambda}}=$ \( A
\otimes C \stackrel{A \otimes\delta_{\mathbb{C}}}{\longrightarrow}
A \otimes C \otimes C \), is an $A$-coring. Since, for any $V \in
\V_{\mathbb{A}},$ $V \otimes_A (A \otimes C) \simeq V \otimes C,$
the comonad $\bar{\textbf{G}}$ is isomorphic to the comonad
$\textbf{G}_{(\underline{A \otimes C})_{\lambda}}$. Thus, any
entwining structure $\la$ defines a right $A$-module structure
$\xi_{A\otimes C}$ on $A \otimes C$ such that $(A \otimes C,
\bar{\xi}_{A \otimes C}=m_A \otimes C, \xi_{A \otimes C})$ is an
$A$-$A$-bimodule and the triple $(\underline{A \otimes
C})_{\lambda}=(A \otimes C , \varepsilon_{(\underline{A \otimes
C})_{\lambda}}, \delta_{(\underline{A \otimes C})_{\lambda}})$ is
an $A$-coring. Moreover, when this is the case, the comonad
$\textbf{G}_{(\underline{A \otimes C})_{\lambda}}$ on
$\V_{\mathbb{A}}$ extends the comonad $\textbf{G}_{\mathbb{C}}$.
It follows that $\V_{\mathbb{A}} ^{(\underline{A \otimes
C})_{\lambda}}=\V^{\mathbb{C}}_{\mathbb{A}}(\lambda).$

Conversely, let $\mathbb{A}=(A,e_A,m_A)$ be an algebra and $\C$ a
coalgebra in $\V$, and suppose that $A \otimes C $ has the
structure $\xi_{A \otimes C }$ of a right $A$-module such that the
triple \begin{equation}\xymatrix{ \underline{A \otimes C}=((A
\otimes C , m_A \otimes C, \xi_{A \otimes C }), A \otimes C
\ar[r]^-{A \otimes \varepsilon_{\mathbb{C}}} &A, \,\, A \otimes C
\ar[r]^-{A \otimes \delta_{\mathbb{C}}} &A \otimes C \otimes
C)}\end{equation} is an $A$-coring. Then it is easy to see that
the comonad $\textbf{G}_{\underline{A \otimes C}}$ on
$\V_{\mathbb{A}}$ extends the comonad $\textbf{G}_{\mathbb{C}}$ on
$\V$, and thus defines an entwining structure
$\lambda_{\underline{A \otimes C}}: C \otimes A \to A \otimes C$.

Summarising, we have

\begin{theorem}Let $\mathbb{A}=(A,e_A,m_A)$ be an algebra and $\C$ a
coalgebra in $\V$. Then there exists a bijection between right
$A$-module structures $\xi_{A \otimes C }$ making $(A \otimes C ,
m_A \otimes C, \xi_{A \otimes C })$ an $A$-bimodule for which the
triple (1) is an $A$-coring and entwining structures $(\mathbb{C},
\mathbb{A}, \lambda)$, given by:
$$ \xymatrix{
\xi_{A \otimes C} \ar[r] &(\lambda_{\underline{A \otimes C}} : C
\otimes A  \ar[rr]^-{e_A \otimes C \otimes A}&& A \otimes C
\otimes A \ar[r]^-{\xi_{A \otimes C}}& A \otimes C)}$$ with
inverse given by $$ \xymatrix{\lambda \ar[r]& (\xi_{A \otimes C} :
A \otimes C \otimes A \ar[r]^-{A \otimes \lambda}& A \otimes A
\otimes C \ar[rr]^-{m_A \otimes C}&& A \otimes C)}$$

Under this equivalence $\V_{\mathbb{A}} ^{(\underline{A \otimes
C})_{\lambda}}=\V_{\mathbb{A}}^{\mathbb{C}}(\lambda)$.

\end{theorem}

\section{Some categorical results}

Let $\G$ be a comonad on a category $\A$, and let $U_{\textbf{G}}:
A_{\textbf{G}} \to \A$ be the forgetful functor. Fix a functor $F
: \B \to \A$, and consider a functor $\bar{F} : \B \to
A_{\textbf{G}}$ making the diagram

\begin{equation}
\xymatrix{ \B \ar[rr]^-{\bar{F}} \ar[dr]_{F} && A_{{\textbf{G}}}
\ar[dl]^{U_{\textbf{G}}}\\
& \A & } \end{equation} commutative. Then $\bar{F} (b)=(F(b),
\alpha_{F(b)})$ for some $\alpha_{F(b)} : F(b) \to GF(b)$.
Consider the natural transformation
\begin{equation} \bar{\alpha}_F : F \to GF,
\end{equation} whose $b$-component is $\alpha_{F(b)}$.

It is proved in \cite {D} that:

\begin{theorem} Suppose that $F$ has a right adjoint
$R : \A \to \B$ with unit $ \eta: 1 \to FU$ and counit
$\varepsilon : FU \to 1$. Then the composite
$$\xymatrix{t_{\bar{F}}:
FU \ar[r]^-{\bar{\alpha}_{_F} U} & GFU \ar[r]^{G \varepsilon}
&G.}$$ is a morphism from the comonad $\textbf{G}'=(FU,
\varepsilon, F \eta U)$ generated by the adjunction $\eta,
\varepsilon : F \dashv U : \B \to \A$ to the comonad $\textbf{G}$.
Moreover, the assignment
$$\bar{F} \longrightarrow t_{\bar{F}}$$ yields a one to one
correspondence between functors $\bar{F} : \B \to \A_G$ making the
diagram (2) commutative and morphisms of comonads $t_{\bar{F}} :
\textbf{G}' \to \textbf{G}$.
\end{theorem}

Write $\beta_U $ for the composite $\xymatrix{ U \ar[r]^-{\eta U}& UFU \ar[r]^-{U t_{\bar{F}}} & UG}.$

\begin{proposition}
The equalizer $\bar{U}$, if it exists, of the following diagram $$\xymatrix{ UU_G
\ar@{->}@<0.5ex>[rr]^-{UU_G \eta_G} \ar@ {->}@<-0.5ex>
[rr]_-{\beta_U U_G}&& UGU_G=UU_GF_GU_G,}$$ where $\eta_G : 1 \to F_GU_G$ is the unit of the adjunction $U_G \dashv F_G,$ is right adjoint to $\overline{F}$.

\end{proposition}
\begin{proof}
See \cite{B} or \cite{D}.
\end{proof}

Let $\bar{F} : \B \to \A_{\textbf{G}}$ be a functor making (2)
commutative and let $t_{\bar{F}} : \textbf{G}' \to \textbf{G}$ be
the corresponding morphism of comonads.  Consider the following
composition
$$\xymatrix{
\B \ar[r]^{K_{\textbf{G}'}} & \A_{\textbf{G}'}
\ar[r]^{\A_{t_{\bar{F}}}}& \A_{\textbf{G}},}
$$ where
\begin{itemize}

\item $K_{\textbf{G}'} : \B \to \A_{\textbf{G}'}, \,\,
K_{\textbf{G}'}(b)=(F(b), F(\eta_b))$ is the Eilenberg-Moore
comparison functor for the comonad $\textbf{G}'$.

\item $A_{t_{\bar{F}}}$ is the functor $$((a, \theta_a) \in
\A_{\textbf{G}}') \longrightarrow ((a, (t_{\bar{F}})_a \cdot
\theta_a) \in \A_{\textbf{G}})$$ induced by the morphism of
comonads $t_{\bar{F}} : \textbf{G}' \to \textbf{G}.$
\end{itemize}

\begin{lemma} The diagram
\begin{equation}
\xymatrix{ \B \ar[r]^{K_{G'}} \ar[dr]_{{\bar{F}}} &
\A_{G'} \ar[d]^{\A_{{t_{\bar{F}}}}}\\
& \A_{G}}
\end{equation} is commutative.
\end{lemma}

\begin{proof} Let $b \in \B$. Then $K_{\textbf{G}'}(b)=(F(b),
F(\eta_b))$ and $\A_{t_{\bar{F}}}(F(b), F(\eta_b))=(F(b),
(t_{\bar{F}})_{F(b)} \cdot F(\eta_b)).$ Since
$(t_{\bar{F}})_{F(b)}$ is the composite
$$ \xymatrix{FUF(b) \ar[rr]^-{(\bar{\alpha}_F)_{UF(b)}} && GFUF(b)
\ar[r]^-{G \varepsilon _{F(b)}} & GF(b),}$$ and since by
naturality of $\bar{\alpha}_F$, the diagram
$$\xymatrix{
F(b) \ar[r]^{(\bar{\alpha})_b} \ar[d]_{F(\eta_b)} & GF(b)
\ar[d]^{GF(\eta_b)}\\
FUF(b) \ar[r]_{(\bar{\alpha})_{UF(b)}} & GFUF(b)}$$ commutes, we
have $$(t_{\bar{F}})_{F(b)} \cdot
F(\eta_b)=G(\varepsilon_{F(b)})\cdot (\bar{\alpha}_F)_{UF(b)}
\cdot F(\eta_b)=G (\varepsilon_{F(b)}) \cdot GF (\eta_b) \cdot
(\bar{\alpha}_F)_b=(\bar{\alpha}_F)_b=\alpha_{F(b)}.
$$ Thus $$(\A_{t_{\bar{F}}} \cdot K_{\textbf{G}'}
)(b)=\A_{t_{\bar{F}}}(K_{\textbf{G}'}(b))= \A_{t_{\bar{F}}} (F(b),
F(\eta_b))= (F(b), (t_{\bar{F}})_{F(b)}\cdot F(\eta_b))=(F(b),
\alpha_{F(b)}),$$ which just means that $\A_{t_{\bar{F}}} \cdot
K_{\textbf{G}'}=\bar{F}$.
\end{proof}

We are now ready to prove the following

\begin{theorem} Let $\textbf{G}$ be a comonad on a category $\A$,
$\eta, \varepsilon : F \dashv U :\B \to \A$ an adjunction and
$\bar{F} :\B \to \A_G$ a functor with $U_G \cdot \bar{F} =F$. Then
the following are equivalent:
\begin{itemize}
\item [\emph{(i)}]The functor $\bar{F}$ is an equivalence. \item
[\emph{(ii)}]The functor $F$ is comonadic and the morphism of
comonads $$t_{\bar{F}}: \textbf{G}\,'=(FU, \varepsilon, F \eta U)
\to \textbf{G}$$ is an isomorphism.
\end{itemize}
\end{theorem}

\begin{proof} Suppose that $\bar{F}$ is an equivalence of categories.
Then $F$ is isomorphic to the comonadic functor $U_{\textbf{G}}$
and thus is comonadic. Hence the comparison functor
$K_{\textbf{G}'}: \B \to \A_{\textbf{G}'}$ is an equivalence and
it follows from the commutative diagram (4) that
$\A_{t_{\bar{F}}}$ is also an equivalence, and since the diagram

$$\xymatrix{
\A_{G'} \ar[rr]^{\A_{t_{\bar{F}}}} \ar[rd]_{U_{G'}} && \A_G \ar[dl]^{U_G}\\
& \A &}$$ is commutative, $t_{\bar{F}}$ is an isomorphism of
comonads. So $(i) \Longrightarrow (ii)$.

Suppose now that $t_{\bar{F}}: \textbf{G}\,' \to \textbf{G}$ is an
isomorphism of comonads and $F$ is comonadic. Then
\begin{itemize}
\item $K_{\textbf{G}'}$ is an equivalence, since $F$ is comonadic.
\item $\A_{t_{\bar{F}}} $ is an equivalence, since $t_{\bar{F}}$
is an isomorphism.
\end{itemize} And it now follows from the commutative diagram (4)
that $\bar{F}$ is also an equivalence. Thus $(ii) \Longrightarrow
(i)$. This completes the proof of the theorem.

\end{proof}

\begin{remark} In \cite{Go}, J. G\'{o}mez-Torrecillas has proved
that $\bar{F}$ is an equivalence of categories iff $t_{\bar{F}}$
is an isomorphism of comonads, $F$ is conservative, and for any
$(X,x)\in \A_{\textbf{G}}$, $F$ preserves the equalizer of the
pair of parallel morphisms
\begin{equation} \xymatrix@1{U(X) \ar@/^/@<+1ex>[rrr]^{U(x)}
\ar[r]_-{\eta_{U(X)}}&UG'(X) \ar[rr]_{U((t_{\bar{F}})_X)} && UG(X)
\,.}\end{equation} When $t_{\bar{F}}$ is an isomorphism of
comonads, to say that $F$ preserves the equalizer of the pair of morphisms $(5)$ is to
say that $F$ preserves the equalizer of the pair of morphisms
$$\xymatrix{ U(X) \ar@{->}@<0.5ex>[rr]^-{\eta_{U(X)}} \ar@
{->}@<-0.5ex> [rr]_-{U((t^{-1}_{\bar{F}})_X) \cdot U(x)}&&
UG'(X),}$$ which we can rewrite as \begin{equation}\xymatrix{ U(X)
\ar@{->}@<0.5ex>[rr]^-{\eta_{U(X)}} \ar@ {->}@<-0.5ex>
[rr]_-{U((t^{-1}_{\bar{F}})_X \,\cdot \,x)}&&
UG'(X)=UFU(X).}\end{equation} Since $t_{\bar{F}}$ is an
isomorphism of comonds, $\A_{t_{\bar{F}}}$ is an equivalence of
categories, and thus each object $(X,x') \in \A_{\textbf{G}'}$ is
isomorphic to the $\textbf{G}'$-coalgebra $(X,
(t^{-1}_{\bar{F}})_X \,\cdot \,x)$, where $(X,x)\in
\A_{\textbf{G}}$. It follows that when $t_{\bar{F}}$ is an
isomorphism of comonds, to say that $F$ preserves the equalizer of
$(5)$ for each $(X,x)\in \A_{\textbf{G}}$ is to say that $F$
preserves the equalizer of $(6)$ for each $(X, x') \in
\A_{\textbf{G}'}$. Thus, when $t_{\bar{F}}$ is an isomorphism of
comonds, $\bar{F}$ is an equivalence of categories iff $F$ is
conservative and preserves the equalizer of $(6)$ for each $(X,
x') \in \A_{\textbf{G}'}$, which according to (the dual of) Beck's
theorem (see \cite{M}), is to say that the functor $F$ is
comonadic. Hence our theorem 4.4 is equivalent to Theorem 1.7 of
\cite{Go}.
\end{remark}

\section{Some applications}

Let $(\mathbb{C}, \mathbb{A}, \lambda)$ be an entwining structure in a monoidal category $\V=(V, \otimes, I)$,
and let $g: I \to C$ be a group-like element of $\mathbb{C}$.
(Recall that a morphism $g: I \to C$ is said to be a group-like
element of $\mathbb{C}$ if the following diagrams
$$
\xymatrix{ I \ar[r]^{g} \ar@{=}[dr]^{(1)}& C
\ar[d]^{\varepsilon_{\mathbb{C}}}&& I \ar@{}[dr]^{(2)}\ar[r]^{g}
\ar[dr]_{g\otimes g}& C
\ar[d]^{\delta_{\mathbb{C}}}\\
& I && & C \otimes C}$$ are commutative.)

\begin{proposition} If $\mathbb{C}$ has a group-like element $g: I \to
C$, then $A$ is a right $C$-comodule through the morphism
$$\xymatrix{g_A : A  \ar[r]^-{g \otimes A}& C
\otimes A \ar[r]^{\lambda}& A \otimes C.}$$
\end{proposition}

\begin{proof} Consider the diagram
$$\xymatrix{
A  \ar@{=}[dr] \ar[r]^-{g \otimes A}& C \otimes A \ar[r]^{\lambda}
\ar[d]^{\varepsilon_{\mathbb{C}} \otimes A}&
A \otimes C \ar[d]^{A \otimes \varepsilon_{\mathbb{C}}}\\
& A  \ar@{=}[r]& A\, .}$$ The triangle is commutative by (1) of
the definition of $g$ and the square is commutative by the
definition of $\lambda$ (see the second commutative diagram in the
definition of entwining structures).

Now, we have to show that the following diagram

$$\xymatrix{
A \ar[d]_{g \otimes A}  \ar[r]^{g \otimes A}& C \otimes A
\ar[r]^{\lambda}& A \otimes C \ar[dd]^-{A \otimes
\delta_{\mathbb{C}}}\\
C \otimes A \ar[d]_{\lambda}\\
A \otimes C  \ar[r]_-{g \otimes A \otimes C} & C \otimes A \otimes
C \ar[r]_{\lambda\otimes C} & A \otimes C \otimes C }$$ is also
commutative, which it is since $$(A \otimes
\delta_{\mathbb{C}})\lambda=(\lambda \otimes C)(C \otimes
\lambda)(\delta_{\mathbb{C}} \otimes A)$$ by the definition of
$\lambda$ and since the diagram (2) of definition of group-like
elements is commutative.
\end{proof}

Suppose now that $\V$ admits equalizers. For any $(M, \alpha_M)
\in \V ^{\mathbb{C}}$, write $((M, \alpha_M)^{\mathbb{C}}, i_M)$
for the equalizer of the morphisms

$$\xymatrix{(M, \alpha_M)^{\mathbb{C}} \ar[r]^-{i_M}&
M \ar@{->}@<0.5ex>[r]^-{\alpha_M} \ar@ {->}@<-0.5ex> [r]_-{Mg}& M
\otimes C.}$$

\begin{proposition} $A ^{\mathbb{C}}=(A, g_A)^{\mathbb{C}}$ is an
algebra in $\V$ and $i_A:A ^{\mathbb{C}} \to A $ is an algebra
morphism.
\end{proposition}

\begin{proof} Consider the diagram
\begin{equation}\xymatrix@1{A ^{\mathbb{C}} \ar[r]^{i_A}&
A\ar@/^/@<+1ex>[rr]^{A \otimes g}
\ar[r]_{g \otimes  A}& C \otimes A \ar[r]_{\lambda} & A \otimes C  \\
I \ar@{.>}[u]^{e_{A^{\mathbb{C}}}}\ar[ru]_{e_A}&&}\end{equation}
Since $$g \otimes -: 1_{\V}=I \otimes - \to C \otimes -$$ is a
natural transformation, the diagram
$$\xymatrix{I \ar[d]_{e_A}\ar[r]^{g} & C \ar[d]^{C \otimes  e_A}\\
A \ar[r]_{g \otimes A}& C \otimes A}$$ is commutative. Similarly,
since $e_A \otimes - : 1_{\V}=I \otimes - \to C \otimes -$ is a
natural transformation, the following diagram is also commutative:
$$\xymatrix{I \ar[r]^-{e_A}
\ar[d]_{g}& A \ar[d]^{A \otimes g}\\
C \ar[r]_-{e_A  \otimes C}& A \otimes C\,.}$$ Now we have:
$$\lambda  (g \otimes A)  e_A=\lambda  (C \otimes e_A)  g= \,\,
\text{by the definition of} \,\, \lambda$$$$=(e_A
 \otimes C) g= (A \otimes g)  e_A.$$ Thus there exists a unique morphism
$e_{A^{\mathbb{}}}: I \to A ^{\mathbb{C}}$ for which $i_{A}\cdot
e_{A ^{\mathbb{C}}}=e_A.$

Since
\begin{itemize}
\item  the diagram $$\xymatrix{A \otimes A \ar[d]_{m_A}\ar[r]^-{g
\otimes A \otimes A}
& C \otimes A \otimes A \ar[d]^{C \otimes m_A}\\
A \ar[r]_{g \otimes A}& C \otimes A}$$ is commutative by
naturality of $g \otimes -$; \item $\lambda  (C \otimes m_A)=(m_A
\otimes C) (A \otimes \lambda)  (\lambda \otimes A)$ by the
definition of $\lambda$; \item $\lambda  (g \otimes A)  i_A=(A
\otimes g)  i_A$, since $i_A$ is an equalizer of $\lambda (g
\otimes A)$ and $A \otimes g$; \item the diagram
$$\xymatrix{A \otimes A \ar[d]_{m_A}\ar[r]^-{A \otimes A \otimes g
}
&  A \otimes A \otimes C \ar[d]^{m_A \otimes C}\\
A \ar[r]_{A \otimes g}& A \otimes C}$$ is commutative by
naturality of $m_A \otimes -$,
\end{itemize} we have $$\lambda  (g \otimes A) m_A
(i_A \otimes i_A)=\lambda  (C \otimes m_A)  (g \otimes A \otimes
A)  (i_A \otimes i_A)=$$$$=(m_A \otimes C)(A \otimes
\lambda)(\lambda \otimes A)(g \otimes A \otimes A)(i_A \otimes
i_A)=$$$$=(m_A \otimes C) (A \otimes \lambda)(A \otimes g \otimes
A) (i_A \otimes i_A)=(m_A \otimes C)(A \otimes A \otimes g)(i_A
\otimes i_A)=$$$$=(A \otimes g) m_A (i_A \otimes i_A).$$ Thus the
morphism $m_A \cdot (i_A \otimes i_A)$ equalizes the morphisms
$\lambda \cdot (g \otimes A)$ and $A \otimes g$, and hence there
is a unique morphism
$$m_{A^{\mathbb{C}}}: A^{\mathbb{C}} \otimes A^{\mathbb{C}} \to
A^{\mathbb{C}}$$ such that the diagram
\begin{equation}\xymatrix{A^{\mathbb{C}}\otimes A^{\mathbb{C}}
\ar[d]_{m_{A^{\mathbb{C}}}} \ar[r]^-{i_A \otimes i_A }
&  A \otimes A \ar[d]^{m_A }\\
A^{\mathbb{C}} \ar[r]_{i_A}& A}\end{equation} commutes. It is now
straightforward to show that the triple $(A^{\mathbb{C}},
e_{A^{\mathbb{C}}}, m_{A^{\mathbb{C}}})$ is an algebra in $\V$;
moreover, the triangle of the diagram $(7)$ and the diagram $(8)$
show that $i_A$ is an algebra morphism.
\end{proof}

\begin{proposition} $(A, m_A , g_A) \in \V ^{\mathbb{C}}_{\mathbb{A}}(\lambda).$
\end{proposition}

\begin{proof} Since $(A,m_A) \in \V_{\mathbb{A}}$ and $(A, g_A) \in \V
^{\mathbb{C}}$, it only remains to show that the following diagram
is commutative:

\begin{equation}
\xymatrix{ A \otimes A \ar[r]^-{g_A \otimes A} \ar[d]_{m_A} & A
\otimes C \otimes A \ar[r]^{A \otimes \lambda} & A \otimes A
\otimes C
\ar[d]^{m_A \otimes C}\\
A \ar[rr]_{g_A} && A \otimes C.} \end{equation} By the definition
of $g_A$, we can rewrite it as

$$\xymatrix{
A \otimes A   \ar[d]_{m_A} \ar[r]^-{g \otimes A \otimes A}& C
\otimes A \otimes A \ar[r]^-{\lambda \otimes A} \ar@{.>}[d]^{C
\otimes m_A} & A \otimes C \otimes A \ar[r]^{A \otimes \lambda}& A
\otimes A \otimes C \ar[d]^{m_A \otimes C}\\
A \ar[r]_{g \otimes A} & C \otimes A \ar[rr]_{\lambda} && A
\otimes C.}$$ But this diagram is commutative, since

\begin{itemize}
\item the middle square commutes because of naturality of $g
\otimes -$; \item the right square commutes because of the
definition of $\lambda$.

\end{itemize}
\end{proof}

The algebra morphism $i_A:A ^{\mathbb{C}} \to A $ makes $A$ an
$A^{\mathbb{C}}$-$A^{\mathbb{C}}$-bimodule and thus induces the
\emph{extension-of-scalars} functor $$F_{i_A} : \V_{A
^{\mathbb{C}}} \to \V_{A}$$$$(X, \rho_X) \longrightarrow (X
\otimes_{A ^{\mathbb{C}}} A, X \otimes_{A ^{\mathbb{C}}} m_A),$$
and the forgetful functor
$$U_{i_A}:\V_{A} \to \V_{A ^{\mathbb{C}}}$$$$(Y, \varrho_Y)
\longrightarrow (Y, \varrho_Y \cdot (Y \otimes i_A)),$$ which is
right adjoint to $F_{i_A}$. The corresponding comonad on $\V_A$
makes $A \otimes_{A ^{\mathbb{C}}} A$ into an $A$-coring with the
following counit and comultiplication: $$ \xymatrix{\varepsilon :
A \otimes_{A^\mathbb{C}} A \ar[r]^-{q} & A \otimes A
\ar[r]^-{m_A}& A,} $$ (where $q$ is the canonical morphism) and $$
\xymatrix{ \delta : A \otimes_{A^\mathbb{C}} A =
A\otimes_{A^\mathbb{C}} {A^\mathbb{C}} \otimes_{A^\mathbb{C}} A
\ar[rr]^-{A \otimes_{A^\mathbb{C}} i_A \otimes_{A^\mathbb{C}} A}&&
A \otimes_{A^\mathbb{C}} A \otimes_{A^\mathbb{C}} A = (A
\otimes_{A^\mathbb{C}} A)_A \otimes (A \otimes_{A^\mathbb{C}}
A).}$$ We write $\underline{A \otimes_{A^\mathbb{C}} A}$ for this
$A$-coring.

\begin{lemma} For any $X \in \V_{A^\mathbb{C}}$, the triple
$$(X \otimes_{A^\mathbb{C}} A, X \otimes_{A^\mathbb{C}}
m_A, X \otimes_{A^\mathbb{C}} g_A)$$ is an object of the category
$\V ^{\mathbb{C}}_{\mathbb{A}}(\lambda).$
\end{lemma}

\begin{proof} Clearly $(X \otimes_{A^\mathbb{C}} A, X \otimes_{A^\mathbb{C}}
m_A) \in \V_{\mathbb{A}}$ and $((X \otimes_{A^\mathbb{C}} A, X
\otimes_{A^\mathbb{C}} g_A) \in \V ^{\mathbb{C}}$. Moreover, by
(9), the following diagram $$ \xymatrix{ X \otimes_{A^\mathbb{C}}
X \otimes_{A^\mathbb{C}} A \otimes A \ar[rr]^-{X
\otimes_{A^\mathbb{C}} g_A \otimes A} \ar[d]_{X
\otimes_{A^\mathbb{C}} m_A} && X \otimes_{A^\mathbb{C}} A \otimes
C \otimes A \ar[rr]^-{X \otimes_{A^\mathbb{C}} A \otimes \lambda}
&& X \otimes_{A^\mathbb{C}} A \otimes A \otimes C \ar[d]^{X
\otimes_{A^\mathbb{C}}m_A \otimes C}\\
X \otimes_{A^\mathbb{C}} A \ar[rrrr]_{X \otimes_{A^\mathbb{C}}
g_A} &&&& X \otimes_{A^\mathbb{C}} A \otimes C} $$ is commutative.
Thus, $(X \otimes_{A^\mathbb{C}} A, X \otimes_{A^\mathbb{C}} m_A,
X \otimes_{A^\mathbb{C}} g_A) \in \V
^{\mathbb{C}}_{\mathbb{A}}(\lambda).$
\end{proof} The lemma shows that the assignment
$$X \longrightarrow (X \otimes_{A^\mathbb{C}} A, X \otimes_{A^\mathbb{C}} m_A,
X \otimes_{A^\mathbb{C}} g_A)$$ yields a functor $$\bar{F}: \V
_{\mathbb{A}} \to \V
^{\mathbb{C}}_{\mathbb{A}}(\lambda)=\V_{\mathbb{A}}
^{(\underline{A \otimes C})_{\lambda}}.$$ It is clear that
$U_{(\underline{A \otimes C})_{\lambda}} \cdot \bar{F}=F_{i_A}$,
where $U_{(\underline{A \otimes C})_{\lambda}}:\V_{\mathbb{A}}
^{(\underline{A \otimes C})_{\lambda}} \to \V_{\mathbb{A}}$ is the
underlying functor. It now follows from Theorem 3.1 that the
composite $$\xymatrix{A \otimes_{A^\mathbb{C}} A \ar[rr]^{A
\otimes g_A} && A \otimes A \otimes C \ar[rr]^-{m_A \otimes C}&& A
\otimes C}$$ is a morphism of $A$-corings $\underline{A
\otimes_{A^\mathbb{C}} A} \to (\underline{A \otimes
C})_{\lambda}.$ We write $\emph{can}$ for this morphism. We say
that $A$ is $(\mathbb{C},g)$-Galois if $\emph{can}$ is an
isomorphism of $A$-corings.

Applying Theorem 4.4 the commutative diagram
$$\xymatrix{\V_{A^{\mathbb{C}}} \ar[r]^-{\bar{F}} \ar[dr]_{F_{i_A}=
- \otimes_{A ^{\mathbb{C}}} A}
 & {\V_{\mathbb{A}}}^{(\underline{A \otimes C})_{\lambda}}
\ar[d]^{U_{(\underline{A \otimes C})_{\lambda}}}\\
& \V_{\mathbb{A}}}$$ we get:

\begin{theorem} Let $(\mathbb{C},\mathbb{A}, \lambda)$ be an
entwining structure, and let $g: I \to C$ be a group-like element
of $\mathbb{C}$. Then the functor $$\bar{F}: \V_{A^\mathbb{C}} \to
\V ^{\mathbb{C}}_{\mathbb{A}}(\lambda)$$ is an equivalence if and
only if $A$ is $(\mathbb{C}, g)$-Galois and the functor $F$ is
comonadic.
\end{theorem}

\bigskip

Let $\Al$ and $\mathbb{B}=(B, e_B , m_B)$ be algebras in $\V$ and
let $M \in {{_\mathbb{A}{\V}}_{\mathbb{B}}}$. We call $_A M$
(resp. $M_B$)
\begin{itemize}
\item \emph{flat}, if the functor $- \otimes_A \! M: {\V_A} \to
\V_B$ (resp. $M \otimes_B - : {_B \V} \to {_A \V}$) preserves
equalizers; \item \emph{faithfully flat}, if the functor $-
\otimes_A \!M : {\V_A} \to \V_B$ (resp. $M\! \otimes_B - : {_B \V} \to
{_A \V}$) is conservative and flat (equivalently, preserves and
reflects equalizers);
\end{itemize}

\begin{theorem} Let $(\mathbb{C},\mathbb{A}, \lambda)$ be an
entwining structure, and let $g: I \to C$ be a group-like element
of $\mathbb{C}$. If $C$ is flat, then the following are equivalent

\begin{itemize}
\item[\emph{(i)}] The functor $$\bar{F}: \V_{A^\mathbb{C}} \to \V
^{\mathbb{C}}_{\mathbb{A}}(\lambda)={\V_{\mathbb{A}}}^{(\underline{A
\otimes C})_{\lambda}}$$ is an equivalence of categories.
\item[\emph{(ii)}] $A$ is $(\mathbb{C}, g)$-Galois and
${_{A^{\mathbb{C}}} A}$ is faithfully flat.
\end{itemize}
\end{theorem}

\begin{proof}Since any left adjoint functor that is conservative and
preserves equalizers is comonadic by a simple and
well-known application (of the dual of) Beck's theorem, one
direction is clear from Therem 5.5; so suppose that $\bar{F}$ is
an equivalence of categories. Then, by Theorem 4.5, $A$ is
$(\mathbb{C}, g)$-Galois and the functor $F_{i_A}$ is comonadic.
Since any comonadic functor is conservative, $F_{i_A}$ is also
conservative. Thus, it only remains to show that
${_{A^{\mathbb{C}}} A}$ is flat.

Since $C$ is flat by our assumption, ${_A (A \otimes C)}$ is also
flat. It follows that the underlying functor of the comonad
$\textbf{G}_{(\underline{A \otimes C})_{\lambda}}$ on $\V _A$
preserves equalizers. We recall (for example, from \cite{B}) that
if $\textbf{G}=(G, \varepsilon_G, \delta_G)$ is a comonad on a
category $\mathcal{A}$, and if $\mathcal{A}$ has some type of
limits preserved by $G$, then the category
$\mathcal{A}_{\textbf{G}}$ has the same type of limits and these
are preserved by the underlying functor
$U_{\textbf{G}}:\mathcal{A}_{\textbf{G}} \to \mathcal{A}$. Thus
the functor $U_{(\underline{A \otimes
C})_{\lambda}}:{\V_{\mathbb{A}}}^{(\underline{A \otimes
C})_{\lambda}}\to \V_{\mathbb{A}}$ preserves equalizers, and since
$\bar{F}$ is an equivalence of categories, the functor $F_{i_A}= -
\otimes_{A ^{\mathbb{C}}} A$ also preserves equalizers, which just
means that ${_{A^{\mathbb{C}}} A}$ is flat. This completes the
proof.
\end{proof}

\bigskip
From now on we suppose at all times that our $\V$ is a strict
braided monoidal category with braiding $\sigma_{X,Y} : X \otimes
Y \to Y \otimes X$. Then the tensor product of two (co)algebras in
$\V$ is again a (co)algebra; the multiplication $m_{A \otimes B}$
and the unit $e_{A \otimes B}$ of the tensor product of two
algebras $\Al$ and $\mathbb{B}=(B, e_\mathbb{B}, m_\mathbb{B})$
are given through
$$m_{A \otimes B}=(m_{A} \otimes
m_B)(A \otimes \sigma_{A,B}\otimes B)$$ and
$$e_{A \otimes B}=e_A \otimes e_B.$$

A bialgebra $\mathbb{H}=(\bar{H}=(H, e_H, m_H),\underline{H}=(H,
\varepsilon_H, \delta_H))$ in $\V$ is an algebra $\bar{H}=(H, e_H,
m_H)$ and a coalgebra $\underline{H}=(H, \varepsilon_H,
\delta_H)$, where $\varepsilon_H$ and $\delta_H$ are algebra
morphisms, or, equivalently, $e_H$ and $m_H$ are coalgebra
morphisms.

A Hopf algebra $\mathbb{H}=(\bar{H}=(H, e_H,
m_H),\underline{H}=(H, \varepsilon_H, \delta_H),S)$ in $\V$ is a
bialgebra $\mathbb{H}$ with a morphism $S : H \to H$, called the
antipode of $\mathbb{H}$, such that $$m_H (H \otimes
S)\delta_H=m_H (S \otimes H)\delta_H.$$

Recall that for any bialgebra $\mathbb{H}$, the category $\V
^{\underline{H}}$ is monoidal: The tensor product $(X,
\delta_X)\otimes (Y, \delta_Y)$ of two right
$\mathbb{H}$-comodules $(X, \delta_X)$ and $ (Y, \delta_Y)$ is
their tensor product $X \otimes Y$ in $\V$ with the coaction
$$\xymatrix{\delta_{X \otimes Y}: X \otimes Y \ar[r]^-{\delta_X \otimes
\delta_Y}& X \otimes H \otimes Y \otimes H \ar[rr]^-{X \otimes \sigma_{X,Y}
\otimes Y}&& X \otimes Y \otimes H \otimes H \ar[rr]^-{X \otimes Y
\otimes m_H} && X \otimes Y \otimes H}.$$ The unit object for this
tensor product is $I$ with trivial $\underline{H}$-comodule
structure $e_H : I \to H.$

\begin{proposition} Let $\mathbb{H}=(\bar{H}=(H, e_H,
m_H),\underline{H}=(H, \varepsilon_H, \delta_H))$ be a bialgebra
in $\V$. For any algebra $\mathbb{A}=(A, e_A, m_A)$ in $\V$, the
following conditions are equivalent: \begin{itemize} \item
$\mathbb{A}=(A, e_A, m_A)$ is an algebra in the monoidal category
$\V ^{{\underline{H}}}$; \item $\mathbb{A}=(A, e_A, m_A)$ is an
$H$-comodule algebra; that is, $A$ is a right $H$-comodule and the
$H$-comodule coaction $\alpha_A : A \to A \otimes H$ is a morphism
of algebras in $\V$ from the algebra $\Al$ to the algebra $A
\otimes \bar{H}=(A \otimes \bar{H}, e_A \otimes e_{H}, m_{A
\otimes \bar{H}})$.

\end{itemize}
\end{proposition}

Suppose now that $\mathbb{A}=(A, e_A, m_A)$ is a right
$H$-comodule algebra with $H$-coaction $\alpha_A : A \to A \otimes
H$. By the previous proposition, $A$ is an algebra in the monoidal
category $\V ^{{\underline{H}}}$, and thus defines a monad
$\textbf{T}^{A}_H=(T^A_H, \eta^A_H, \mu^A_H)$ on $\V
^{{\underline{H}}}$ as follows: \begin{itemize} \item $T^A_H(X,
\delta_X)=(X, \delta_X) \otimes (A,\alpha_A)$; \item
${(\eta^A_H)}_{(X, \delta_X)}=X \otimes e_A;$ \item
${(\mu^A_H)}_{(X, \delta_X)}=X \otimes m_A$.
\end{itemize} It is easy to see that the monad $\textbf{T}^{A}_H$
extends the monad $\textbf{T}^{A}$; and it follows from Theorem
2.1 that there exists a distributive law $\lambda_{\alpha}:
\textbf{T}^{\mathbb{A}}\cdot \textbf{G}_{\underline{H}} \to
\textbf{G}_{{\underline{H}}}\cdot \textbf{T}^{\mathbb{A}}$ from
the monad $\textbf{T}^{\mathbb{A}}$ to the comonad
$\textbf{G}_{\underline{H}}$, and hence an entwining structure
$(\underline{H}, \mathbb{A}, \lambda_{(A, \alpha_A)})$, where
$\lambda_{(A, \alpha_A)}=(\lambda_{\alpha})_I$.

Therefore we have:

\begin{theorem}Every right $\mathbb{H}$-comodule algebra $\mathbb{A}=
((A, \alpha_A), m_A, e_A)$ defines an entwining structure
$(\underline{H}, \mathbb{A}, \lambda_{(A, \alpha_A)}: C \otimes A
\to A \otimes C)$.
\end{theorem}

\begin{proposition}Let $\mathbb{A}=
((A, \alpha_A), m_A, e_A)$ be a right  $\mathbb{H}$-comodule
algebra. Then the entwining structure $\lambda_{A, \alpha_A}: H
\otimes A \to A \otimes H$ is given by the composite: $$\xymatrix{
H \otimes A \ar[r]^-{H \otimes \alpha_A} & H \otimes A \otimes H
\ar[r]^-{\sigma_{H,A}\otimes H} & A \otimes H \otimes H \ar[r]^-{A
\otimes m_H}& A \otimes H}.$$
\end{proposition}

\begin{proof} Since $(A \otimes \alpha_A)\,\, , (H, \delta_H) \in
\V^{\underline{H}},$ the pair $(A \otimes H, \delta_{A \otimes
H})$, where $\delta_{A \otimes H}$ is the composite $$\xymatrix{H
\otimes A \ar[r]^-{\delta_H \otimes \alpha_A}& H \otimes H \otimes
H \otimes A \ar[rr]^-{H \otimes \sigma_{H,A}\otimes H}&& H \otimes
A \otimes H \otimes H \ar[rr]^-{H \otimes A \otimes m_H} && H
\otimes A \otimes H\, ,}$$ is also an object of
$\V^{\underline{H}}$, and it follows from Theorem 1.1 that
$\lambda_{(A, \alpha_A)}$ is the composite $$\xymatrix{H \otimes A
\ar[r]^-{\delta_{A \otimes H}}& H \otimes A \otimes H
\ar[rr]^-{\varepsilon_H \otimes A \otimes H} && A \otimes H .}
$$ Consider now the following diagram

$$\xymatrix{
 H \otimes A \otimes H
\ar[rr]^-{\delta_H \otimes A \otimes H} \ar@{=}[ddrr]&& {H \otimes
H} {\otimes A  \otimes H} \ar[rr]^-{H \otimes \sigma_{H,A} \otimes
H} \ar[dd]_{\varepsilon_H \otimes H \otimes A \otimes H}&& H
\otimes A \otimes H \otimes H \ar[rr]^-{H \otimes A \otimes m_H}
\ar[dd]_{\varepsilon_H \otimes A \otimes H \otimes H} &&H \otimes
A
\otimes H \ar[dd]_{\varepsilon_H \otimes A \otimes H}\\\\
H \otimes A \ar[uu]^-{H \otimes \alpha_A}&& H \otimes A \otimes H
\ar[rr]_-{\sigma_{H,A}\otimes H} && A \otimes H \otimes H
\ar[rr]_{A \otimes m_H} && A \otimes H \,.}$$ Since in this
diagram
\begin{itemize}
\item the triangle commutes because $\varepsilon_H$ is the counit
for $\delta_H$; \item the left square commutes by naturality of
$\sigma$; \item the right square commutes because $- \otimes -$ is
a bifunctor,
\end{itemize} it follows that $$\lambda_{(A, \alpha_A)}=
(A \otimes m_H)(\sigma_{H,A} \otimes H)(H \otimes \alpha_A).$$
\end{proof}

Note that the morphism $e_H : I \to H$ is a group-like element for
the coalgebra $\underline{H}=(H, \varepsilon_H , \delta_H)$.

\begin{proposition} Let $\mathbb{H}=(\bar{H}=(H, e_H,
m_H),\underline{H}=(H, \varepsilon_H, \delta_H))$ be a bialgebra
in $\V$, and let $\mathbb{A}=((A,\alpha_A), e_A , m_A)$ be a right
$\mathbb{H}$-comodule algebra. Then the right
$\underline{H}$-comodule structure on $A$ corresponding to the group-like
element $e_H :I\to H$ as in Proposition 4.1 coincides with
$\alpha_A$.
\end{proposition}

\begin{proof} We have to show that $$(A \otimes m_H)(\sigma_{H,A}
\otimes H)(H \otimes \alpha_A)(e_H \otimes A)=\alpha_A.$$ But
since
\begin{itemize}
\item clearly $(H \otimes \alpha_A)(e_H \otimes A)=(e_H \otimes A
\otimes H) \cdot \alpha_A;$ \item $(\sigma_{H,A} \otimes H) \cdot
(e_H \otimes A \otimes H)=A \otimes e_H \otimes H \,\,\text{by
naturality of}$\,\, $\sigma;$ \item $(A \otimes m_H ) \cdot (A
\otimes e_H \otimes H)=1_{A \otimes H} \,\,\text{since}\,\, e_H
\,\,\text{is the identity for}$\,\, $m_H$,
\end{itemize} we have that $$(A \otimes m_H)(\sigma_{H,A}
\otimes H)(H \otimes \alpha_A)(e_H \otimes A)=$$$$=(A \otimes
m_H)(\sigma_{H,A} \otimes H)(e_H \otimes A \otimes
H)\alpha_A=$$$$=(A \otimes m_H)(A \otimes e_H \otimes
H)\alpha_A=$$$$=1_{A \otimes H}\cdot \alpha_A=\alpha_A.$$
\end{proof}
It now follows from Proposition 5.3 that

\begin{proposition} $\mathbb{A}=(A, e_A, m_A) \in \V_{\mathbb{A}}^
{\underline{H}}(\lambda_{A, \alpha_A}).$
\end{proposition}

Recall that for any $(X, \alpha_X)\in \V ^{\underline{H}}$, the
algebra $X^{\underline{H}}=(X, \alpha_X)^{\underline{H}}$ is the
equalizer of the morphisms

$$\xymatrix{
X \ar@{->}@<0.5ex>[r]^-{\alpha_X} \ar@ {->}@<-0.5ex> [r]_-{X
\otimes e_H}& X \otimes H.}$$

Applying Theorem 5.5 we get

\begin{theorem}Let $\mathbb{H}=(\bar{H}=(H, e_H,
m_H),\underline{H}=(H, \varepsilon_H, \delta_H))$ be a bialgebra
in $\V$, let $\mathbb{A}=((A,\alpha_A), e_A , m_A)$ be a right
$\mathbb{H}$-comodule algebra, and let $\lambda_{(A, \alpha_A)}:H
\otimes A \to A \otimes H$ be the corresponding entwining
structure. Then the functor $$\bar{F}: \V_{A^{\underline{H}}} \to
\V_{\mathbb{A}}^{\underline{H}}(\lambda_{(A, \alpha_A)})$$
$$(X, \nu_X)\longrightarrow
(X \otimes_{\mathbb{A}^{\underline{H}}} A, X
\otimes_{\mathbb{A}^{\underline{H}}} m_A, X
\otimes_{\mathbb{A}^{\underline{H}}} \alpha_A)$$ is an equivalence
of categories iff the extension-of-scalars functor
$$F_{i_A}: \V_{A^{\underline{H}}} \to \V_{A}$$
$$(X, \nu_X)\longrightarrow (X \otimes_{\mathbb{A}^{\underline{H}}}
A, X \otimes_{\mathbb{A}^{\underline{H}}} m_A)$$ is comonadic and
$A$ is $\underline{H}$-Galois (in the sense that the canonical
morphism
$$\emph{can}:A \otimes_{A^{\underline{H}}} A \to A \otimes
H$$ is an isomorphism).
\end{theorem}

Now applying Theorem 5.6 we get

\begin{theorem} Let $\mathbb{H}=(\bar{H}=(H, e_H,
m_H),\underline{H}=(H, \varepsilon_H, \delta_H))$ be a bialgebra
in $\V$, let $\mathbb{A}=((A,\alpha_A), e_A , m_A)$ be a right
$\mathbb{H}$-comodule algebra, and let $\lambda_{(A, \alpha_A)}:H
\otimes A \to A \otimes H$ be the corresponding entwining
structure. Suppose that $H$ is flat. Then the following are
equivalent:
\begin{itemize}
\item[\emph{(i)}] The functor $$\bar{F}: \V_{A^{\underline{H}}}
\to \V_{\mathbb{A}}^{\underline{H}}(\lambda_{(A, \alpha_A)})$$
$$(X, \nu_X)\longrightarrow
(X \otimes_{\mathbb{A}^{\underline{H}}} A, X
\otimes_{\mathbb{A}^{\underline{H}}} m_A, X
\otimes_{\mathbb{A}^{\underline{H}}} \alpha_A)$$ is an equivalence
of categories.

\item[\emph{(ii)}] $A$ is $\underline{H}$-Galois and
${_{\mathbb{A}^{\underline{H}}} A}$ is faithfully flat.
\end{itemize}
\end{theorem}

\bigskip

Let $\mathbb{H}=(\bar{H}=(H, e_H, m_H),\underline{H}=(H,
\varepsilon_H, \delta_H))$  be a bialgebra in $\V$, and let
$\mathbb{A}=((A,\alpha_A), e_A , m_A)$ be a right
$\mathbb{H}$-comodule algebra. A right $(\mathbb{A},
\mathbb{H})$-module is a right $A$-module which is a right
$\underline{H}$-comodule such that the $\underline{H}$-comodule
structure morphism is a morphism of right $A$-modules. Morphisms
of right $(\mathbb{A}, \mathbb{H})$-modules are right $A$-module
right $\underline{H}$-comodule morphisms. We write $\V
^{\mathbb{H}}_\mathbb{A}$ for this category. Note that the
category $\V ^{\mathbb{H}}_\mathbb{A}$ is the category
$(\V^{\underline{H}})_{\mathbb{A}}$ of right $\mathbb{A}$-modules
in the monoidal category $\V^{\underline{H}}$, and it follows from
Theorem 2.1 that

\begin{proposition} $\V^{\mathbb{H}}_\mathbb{A}=\V_{\mathbb{A}}^
{\underline{H}}(\lambda_{(A, \alpha_A)})$.
\end{proposition}

The following is an immediate consequence of Theorem 5.12.

\begin{theorem} Let $\mathbb{H}=(\bar{H}=(H, e_\mathbb{H},
m_\mathbb{H}),\underline{H}=(H, \varepsilon_\mathbb{H},
\delta_\mathbb{H}))$  be a bialgebra in $\V$, and let
$\mathbb{A}=((A,\alpha_A), e_A , m_A)$ be a right
$\mathbb{H}$-comodule algebra. Then the functor
$$\bar{F}: \V_{A^{\underline{H}}} \to
\V_{\mathbb{A}}^{\mathbb{H}}$$ is an equivalence of categories iff
the extension-of-scalars functor
$$F_{i_A}: \V_{A^{\underline{H}}} \to \V_{A}$$ is comonadic and
$A$ is $\underline{H}$-Galois.
\end{theorem}

\bigskip

Let $\mathbb{H}=(\bar{H}=(H, e_H, m_H),\underline{H}=(H,
\varepsilon_H, \delta_H), S)$  be an Hopf algebra in $\V$. Then
clearly $\bar{H}=(H, e_H, m_H)$ is a right $\mathbb{H}$-comodule
algebra.

\begin{proposition} The composite $$\xymatrix{x: H \otimes H
\ar[r]^-{H \otimes \delta_H} & H \otimes H \otimes H \ar[r]^-{m_H
\otimes H}& H \otimes H}$$ is an isomorphism.
\end{proposition}

\begin{proof} We will show that the composite $$\xymatrix{ y: H \otimes H
\ar[r]^-{H \otimes \delta_H} & H \otimes H \otimes H \ar[r]^-{H
\otimes S \otimes H}& H \otimes H \otimes H \ar[r]^-{m_H \otimes
H}& H \otimes H }$$ is the inverse for $x$. Indeed, consider the
diagram

$$ \xymatrix{
H \otimes H \ar@{}[rrdd]^{(1)} \ar[dd]_-{H \otimes \delta_H}
\ar[rr]^-{H \otimes \delta_H} && H \otimes H \otimes H
\ar@{}[rrdd]^{(2)} \ar[rr]^-{m_H \otimes H} \ar[dd]_-{H \otimes H
\otimes \delta_H} && H \otimes H \ar[dd]^{H \otimes
\delta_H}\\\\
H \otimes H \otimes H  \ar[rr]^-{H \otimes \delta_H \otimes H} &&
H \otimes H \otimes H \otimes H \ar@{}[rrdd]^{(3)}\ar[dd]_{H
\otimes H \otimes S \otimes H} \ar[rr]^{m_H \otimes H \otimes H}&&
H \otimes H \otimes H
\ar[dd]^{H \otimes S \otimes H}\\\\
&&H \otimes H \otimes H \otimes H \ar@{}[rrdd]^{(4)} \ar[dd]_{H
\otimes m_H \otimes H} \ar[rr]^-{m_H \otimes H \otimes H}&& H
\otimes H \otimes H \ar[dd]^{m_H \otimes
H}\\\\
&& H \otimes H \otimes H \ar[rr]_{m_H \otimes H}&& H \otimes H\,
.}$$ We have:
\begin{itemize}
\item Square (1) commutes because of coassociativity of
$\delta_H$; \item Square (2) commutes because of naturality of
$m_H \otimes -$; \item Square (3) commutes because $- \otimes -$
is a bifunctor; \item Square (4) commutes because of associativity
of $m_H$.
\end{itemize}
 Then $$yx=(m_H \otimes H )(H \otimes S \otimes H)(H \otimes
 \delta_H)(m_H \otimes H)(H \otimes \delta_H)=$$

 $$=(m_H \otimes H )(H \otimes m_H \otimes H)
 (H \otimes H \otimes S \otimes H)(H \otimes \delta_H \otimes H)(H \otimes
 \delta_H),$$ but since $$m_H (H \otimes S )\delta_H =e_H \cdot \varepsilon_H,$$
$$yx=(m_H \otimes H)(H \otimes e_H \varepsilon_H \otimes H )(H \otimes
\delta_H)=$$$$=(m_H \otimes H)(H \otimes e_H \otimes H )(H \otimes
 \varepsilon_H \otimes H )(H \otimes
\delta_H)=$$$$=1_{H \otimes H} \otimes 1_{H \otimes H}=1_{H
\otimes H}.$$ Thus $yx=1$. The equality $xy=1$ can be shown in a
similar way.

\end{proof}

\begin{proposition}$(H, \delta_H)^{\underline{H}}\simeq (I, e_H).$
\end{proposition}

\begin{proof} We will first show that the diagram
$$\xymatrix{ H \ar@{=}[d]\ar@{->}@<0.5ex>[rr]^-{H \otimes e_H} \ar@ {->}@<-0.5ex>
[rr]_-{e_H  \otimes H}&& H \otimes H \ar[d]^{x}\\
H \ar@{->}@<0.5ex>[rr]^-{\delta_H} \ar@ {->}@<-0.5ex> [rr]_-{e_H
\otimes H}&& H \otimes H}
$$ is serially commutative. Indeed, we have: $$x  (H \otimes e_H)=
(m_H \otimes H)(H \otimes \delta_H) (H \otimes e_H)=
\,\,\text{since} \,\, \delta_H \,\, \text{is an algebra
morphism}$$$$=(m_H \otimes H)(H \otimes e_H \otimes e_H)=
\,\,\text{since} \,\, e_H \,\, \text{is the unit for } m_H$$$$=H
\otimes e_H;$$

$$x  (e_H  \otimes H)=(m_H \otimes H)(H \otimes \delta_H)
(e_H \otimes  H)=\,\,\text{since} \,\, e_H \,\, \text{is a
coalgebra morphism}$$$$=(m_H \otimes H)(e_H  \otimes
H)\delta_H=1_H  \delta_H=\delta_H.$$ Thus, ${(H,\delta_H,
e_H)}^{\underline{H}}$ is isomorphic to the equalizer of the pair
$(H \otimes e_H, e_H \otimes  H)$. But since $e_H :I \to H$ is a
split monomorphism in $\V$, the diagram
$$
\xymatrix{ I \ar[r]^-{e_H}& H \ar@{->}@<0.5ex>[rr]^-{H \otimes
e_H} \ar@ {->}@<-0.5ex> [rr]_-{e_H \otimes  H}&& H \otimes H}
$$ is an equalizer diagram. Hence $(H, \delta_H, e_H)^{\underline{H}}\simeq (I, e_H).$

\end{proof}

\begin{theorem} Let $\mathbb{H}=(\bar{H}=(H, e_H,
m_H),\underline{H}=(H, \varepsilon_H, \delta_H), S)$ be a Hopf
algebra in $\V$. Then the functor $$\V \to \V
^{\mathbb{H}}_{\mathbb{H}}$$$$V \to V \otimes H$$ is an
equivalence of categories.
\end{theorem}

\begin{proof}It follows from Propositions 5.16 and 5.17 that $H$
is $\underline{H}$-Galois, and according to Theorem 5.12, the
functor $\V \to \V^{\mathbb{H}}_{\mathbb{H}} $ is an equivalence
iff the functor $- \otimes H : \V \to \V _{\bar{H}}$ is comonadic.
But since the morphism $e_H : I \to H$ is a split monomorphism in
$\V$, the unit of the adjunction $F_{e_H} \dashv U_{e_H}$ is a
split monomorphism, and it follows from
3.16 of \cite{Me} that $F_{e_H}$ is comonadic. This completes the
proof.
\end{proof}

\end{document}